\documentclass[a4paper]{amsart}

\usepackage{latexsym, amssymb,amsmath, amsthm,amsfonts}
\newcommand{\N}{\mathbb N}

\newcommand{\ZZ}{\mathbb{Z}}
\newcommand{\kk}{\Bbbk}

\newcommand{\upto}{,\ldots ,}
\DeclareMathOperator{\Lead}{L}
\DeclareMathOperator{\LeadM}{LM}
\DeclareMathOperator{\Pol}{Pol}
\DeclareMathOperator{\charakt}{char}
\def\SL{\operatorname{SL}}

\def\SL2{\operatorname{SL}_{2}(K)}

\def\GL2{\operatorname{GL}_{2}(K)}

\def\INVSL2{$K[V]^{operatorname{SL}_{2}(K)}$}
\def\INVSO2{$K[V]^{operatorname{SO}_{2}(K)}$}
\def\INVGL2{$K[V]^{operatorname{GL}_{2}(K)}$}

\def\GL{\operatorname{GL}}
\def\SL{\operatorname{SL}}

\def\N{\mathbb{N}}

\def\Tr{\operatorname{Tr}}

\newtheorem{Lemma}{Lemma}
\newtheorem{Theorem}[Lemma]{Theorem}

\newtheorem{prop}[Lemma]{Proposition}

\newtheorem*{Corollary of Conjecture}{Corollary of Conjecture}

\theoremstyle{definition}

\theoremstyle{remark}
  \newtheorem{rem}[Lemma]{Remark}

\newtheoremstyle{Acknowledgments}% name
  {}% {\topsep}%      Space above
    {}% {\topsep}%      Space below
     {}%         Body font
     {}%         Indent amount (empty = no indent, \parindent = para indent)
    {\bfseries}% Thm head font
    {}%        Punctuation after thm head
     {.5em}%     Space after thm head: " " = normal interword space;
     {\thmname{#1}\thmnumber{ }\thmnote{ (#3)}}%         Thm head spec (can be left empty, meaning `normal')
\theoremstyle{Acknowledgments}

\title[Vector invariants of permutation groups in characteristic zero]{Vector invariants of permutation groups in characteristic zero}

\author{Fabian Reimers}
 \address{Technische Universit\"at M\"unchen, Zentrum Mathematik - M11, 
Boltzmannstr.~3, 85748 Garching, Germany}
\email{reimers@ma.tum.de}

\author{M\"{u}fit Sezer}
\address{Bilkent University, Department of Mathematics\\
Cankaya, Ankara \\06800 Turkey } \email{sezer@fen.bilkent.edu.tr}

\date{\today}

\begin{document}
\maketitle  
\begin{abstract} We consider a finite permutation group acting naturally on a vector space $V$ over a field $\kk$. A well known theorem of G\"obel asserts that the corresponding ring of invariants $\kk[V]^G$  is generated  by invariants of degree at most~$\binom{\dim V}{2}$. In this note we show that if the characteristic of $\kk$ is zero then the top degree of vector coinvariants~$\kk[V^m]_G$ is also bounded above by~$\binom{\dim V}{2}$, which implies the degree bound~$\binom{\dim V}{2}+ 1$ for the ring of vector invariants~$\kk[V^m]^G$. So G\"obel's bound almost holds for vector invariants in characteristic zero as well. 
\end{abstract}
\section{Introduction}

Let $G$ be a finite group, $\kk$ a field and $V$ a finite dimensional vector space over~$\kk$ on which $G$ acts. The action of $G$ on $V$ induces an action on the symmetric algebra $\kk[V]$ on  $V^*$ given by $gf(v)=f(g^{-1}v)$ for $g\in G$, $f\in \kk[V]$ and $v\in V$.  Let $\kk[V]^G$ denote the ring of invariant polynomials in  $\kk[V]$. This is a finitely generated graded subalgebra of $\kk[V]$ and a central goal in invariant theory is to determine $\kk[V]$ by computing generators and relations. We let $\beta (G,V)$ denote the maximal degree of a polynomial
in a minimal homogeneous generating set for $\kk[V ]^G$.   It is well known by \cite{MR1511848, MR1800251, MR1826990} that $\beta (G,V)\le |G|$ if $|G|\in \kk^*$. If the characteristic of $\kk$ divides $|G|$, then the invariant ring is more complicated and there is no bound that applies to all $V$. But it is possible to bound $\beta (G,V)$ using both $|G|$ and dimension of $V$ (see \cite{MR2811606}). The Hilbert ideal 
$I(G,V)$ is the ideal $\kk[V]_+^G \kk[V]$ in $\kk[V]$  generated by all invariants of positive degree. The algebra of coinvariants $\kk[V ]_G$  is the quotient ring $\kk [V]/I(G,V)$. Both Hilbert ideal and the algebra coinvarints are subjects of interest  as it is possible to extract information about the invariant ring from them. Since~$G$ is finite, $\kk[V ]_G$ is finite dimensional as a vector space and the highest degree in which $\kk[V]_G$ is non-zero is called the top degree of coinvariants.
This degree plays an important role in computing the invariant ring and is closely related to   $\beta (G,V)$ when $|G|\in \kk^*$ (see \cite{MR3282998}).

In this paper we study the case where $G$ is a permutation group acting naturally on $V$  by permuting a fixed basis of $V$. By a  well known theorem of G\"obel \cite{MR1339909}, $\beta (G,V)\le {n\choose 2}$, where $n$ is the dimension of $V$.
This bound applies in all characteristics and it is known to be sharp as for the alternating group $A_n$, we have $\beta (A_n,V) = {n\choose 2}$. Now we consider $m$ direct copies $V^m = V \oplus V  \oplus \ldots \oplus V$ of~$V$ with the action of~$G$ extended diagonally. We show that, if $\kk$ has characteristic zero, the top degree of the coinvariant ring $\kk[V^m ]_G$  is also bounded above by ${n\choose 2}$. Our method relies on polarizing polynomials in the Hilbert ideal $I(G,V)$ and obtaining enough monomials in $I(G,V^m)$ to bound the top degree of  $\kk[V^m]_G$. This implies that $\beta (G,V^m)\le {n\choose 2}+1$.  If polarization of a generating set for  $\kk[V ]^G$  gives a generating set for $\kk[V^m ]^G$, then a generating set for $I(G,V^m)$ can be obtained by polarizing any generating set for  
$I(G,V)$. But in general, one should not expect to get a Gr\"obner basis for $I(G,V^m)$  from a Gr\"obner basis for $I(G,V)$ by polarization.

\section{Polarization and the Hilbert ideal}
In this section we prove that to compute the leading monomial of a polarization of a polynomial it is sufficient to polarize the leading monomial of this polynomial.   We identify $\kk[V]$ with $\kk[x_1, \dots ,x_n]$ and $\kk[V^m]$  with $\kk[x^{(j)}_i \mid i = 1\upto n, j = 1 \upto m]$. We use lexicographic order on $\kk[V^m]$  with 
\[ x^{(1)}_1 > x^{(2)}_1 > \cdots > x^{(m)}_1 > x^{(1)}_2 > x^{(2)}_2 > \cdots > x^{(m)}_2 >  \cdots >  \cdots > x^{(m)}_n\] and the order on 
$\kk[V]$ is obtained by setting $m=1$. For an ideal $I$ we denote the lead term ideal of $I$ with $\Lead (I)$ and the leading monomial of a polynomial $f$ is denoted by $\LeadM (f).$  We introduce extra variables $t_1, \dots , t_m$ and define the algebra homomorphism 
\[ Phi: \kk[V] \to \kk[V^m][t_1 \upto t_m], \quad x_i \mapsto \sum_{j=1}^m x^{(j)}_i t_j.\]
For any $f\in \kk[V]$, we write  $$\Phi(f) = \sum_{(k_1, \dots ,k_m)\in \N^m } f_{k_1, \dots ,k_m} \ t_1^{k_1} \cdot \ldots \cdot t_m^{k_m},$$
with polynomials $f_{k_1, \dots ,k_m} \in \kk[V^m]$. This process is known as polarization and for an $m$-tuple $\underline{k}=(k_1, \dots ,k_m)$ let $\Pol_{\underline{k}}(f)$ denote the coefficient $f_{k_1, \dots ,k_m}$.  We set \[ \Pol(f) = \{ \Pol_{\underline{k}}(f) \mid \underline{k} \in \N^m, f_{\underline{k}}  \} \setminus \{ 0 \} .\]
The importance of polarization for invariant theory comes from the fact that $f \in \kk[V]^G$ implies $\Pol(f) \subseteq \kk[V^m]^G$. In addition, Kohls-Sezer \cite{MR3282998} observed that for every polynomial $f\in I(G,V)$ in the Hilbert ideal we have $\Pol(f)\subseteq I(G,V^m)$.
\begin{Lemma}
 Assume that $\charakt(\kk) = 0$. Let $M, M'$ be monomials in $\kk[V]$ with $M'<M$.   Then for any $\underline{k}\in \N^m$ we have   $$\LeadM (\Pol_{\underline{k}}(M'))<\LeadM (\Pol_{\underline{k}}(M))$$  
 with respect to the lexicographic order fixed above.
\end{Lemma}
\begin{proof}
 Fix $\underline{k}=(k_1, \dots ,k_m)\in \N^m$ and let $M=x_1^{a_1} \cdots  x_n^{a_n} \in \kk[V].$ We have
 $$\Phi(M) =   (x^{(1)}_1 t_1 + \cdots + x^{(m)}_1 t_m)^{a_1} \cdots (x^{(1)}_n t_1 + \cdots + x^{(m)}_n t_m)^{a_n}.$$ 
 Let $h=\LeadM (\Pol_{\underline{k}}(M))$ and write $h= \prod_{1\le i\le n, 1\le j\le m}   \left(x^{(j)}_i\right)^{b_{i,j}}$. Note that $h$ contains with multiplicities $k_1$ variables from the first 
 summand of variables $\{x^{(1)}_1, \dots, x^{(1)}_n\}$. Since $x^{(1)}_1$ is the highest ranked variable among them we have $b_{1,1}=\min \{k_1, a_1\}$. More generally,  $h$ contains with multiplicities $k_j$ variables  from the set $\{x^{(j)}_1, \dots, x^{(j)}_n\}$.   Inductively, multiplicity of the $x^{(j)}_l$ is ${b_{l,j}}$ for $1\le l< i$. Moreover, out of $a_i$ factors $(x^{(1)}_i t_1 + \cdots + x^{(m)}_i t_m)^{a_i}$, $b_{i,l}$ of them contribute $x^{(l)}_i$ to $h$ for $1\le l<j$. Since $x^{(j)}_i$ is the highest rank  monomial in $\{x^{(j)}_i, \dots, x^{(j)}_n\}$, we get a recursive relation
 
\begin{equation}\label{eqHighestMonomExponents}
    b_{i,j} = \min \{k_j - \sum_{l=1}^{i-1} b_{l,j}, \ a_i - \sum_{l=1}^{j-1} b_{i,l}\}. 
\end{equation} 
Note that the coefficient of $h$ in $\Pol_{\underline{k}}(M)$is 
\begin{equation}\label{eqHighestCoeff} \prod_{1\le i\le n}{\frac{a_i!}{b_{i,1}!\cdots b_{i,m}!}}\end{equation}
which is non-zero because  $\charakt(\kk) = 0$.   We may take $M'= x_1^{a_1} \cdots x_{k-1}^{a_{k-1}} \cdot x_k^{a_k'} \cdot \ldots \cdot x_n^{a_n'}$ with $a_{k'}<a_k$. Set 
$h'=\LeadM (\Pol_{\underline{k}}(M'))$ and write $h= \prod_{1\le i\le n, 1\le j\le m}   \left(x^{(j)}_i\right)^{b'_{i,j}}$. As in the case for $b_{i,j}$,  $b'_{i,j}$ depends only on $k_l$ for $1\le l\le j$ and $a_l$ for $1\le l\le i$. Since the multiplicities of the variables $x_1, \dots ,x_{k-1}$ in $M$ and $M'$ are the same, we get that $b_{i,j}=b'_{i,j}$ for $i<k$. On the other hand since $\sum_{1\le l\le m} b_{k,l}=a_k> a_{k'}=\sum_{1\le l\le m} b'_{k,l}$, the equality $  b_{k,l}=  b'_{k,l}$ fails for some $1\le l\le m$. Let $j$ denote the smallest index such that $  b_{k,j}\neq  b'_{k,j}$. Then we have 
\begin{alignat*}{1}
b'_{k,j}&= \min \{k_j - \sum_{l=1}^{k-1} b'_{l,j}, \ a'_k - \sum_{l=1}^{j-1} b'_{k,l}\}\\
&=\min \{k_j - \sum_{l=1}^{k-1} b_{l,j}, \ a'_k - \sum_{l=1}^{j-1} b_{k,l}\}\\
&\le \min \{k_j - \sum_{l=1}^{k-1} b_{l,j}, \ a_k - \sum_{l=1}^{j-1} b_{k,l}\}\\
&=b_{k,j}
\end{alignat*}
So we get that $h'<h$.
\end{proof}
  \begin{rem}
  The coefficient of the highest ranked monomial $h$ which is given by  Equation \ref{eqHighestCoeff} is non-zero over all fields $\kk$ with $a_i!\in \kk^*$ for all $i$. So the assertion of the previous lemma is true for all pairs of monomials $M', M$ with $M=x_1^{a_1} \cdots  x_n^{a_n}$ satisfying $a_i < \charakt(\kk)$ for all $i$. 
  \end{rem}
  \begin{rem}
  Consider the lexicographic monomial order with a slightly different ordering of the variables \begin{equation*}\label{eqLexOrderB} x^{(1)}_1 > x^{(1)}_2 > \ldots > x^{(1)}_n > x^{(2)}_1 > x^{(2)}_2 > \ldots > x^{(2)}_n >  \ldots > \ldots > x^{(m)}_n .
\end{equation*} 
 For a fixed $1\le i\le n$ and $1\le j\le m$, the set of variables in $\{x^{(j)}_1, \dots, x^{(j)}_n\}$ and in  $\{x^{(1)}_i, \dots, x^{(m)}_i\}$ that is smaller than $x^{(j)}_i$ remains unchanged. So the recursive description of the $\LeadM  (\Pol_{\underline{k}}(M))$ in Equation \ref{eqHighestMonomExponents} and consequently the assertion of the lemma carry over to this ordering as well. 
 \end{rem}
  Since $\Pol(f)\subseteq I(G,V^m)$ for all
$f\in I(G,V)$ by \cite[Lemma 12]{MR3282998}, the previous lemma immediately  implies the following.
 \begin{prop}
  Assume that $\charakt(\kk) = 0$. Let $f\in I(G,V)$. Then we have $$\LeadM  (\Pol_{\underline{k}}(\LeadM (f)))\in \Lead (I(G,V^m))$$
  for all $\underline{k}\in \N^m$.
 \end{prop}
 We now prove our main result. 
 \begin{Theorem}  Let $G \leq S_n$ be a permutation group acting naturally on $V = \kk^n$. Assume that $\charakt(\kk)=0$  or $\charakt(\kk)>n$. Then we have $$\beta (G,V^m)\le {n\choose 2}+1.$$
 \end{Theorem}
  \begin{proof}
  A Gr\" obner basis for $I(S_n,V)$ has been computed in \cite{ohsugiwada}. From this source we get that
  $\Lead (I(S_n,V))=(x_1, x_2^2, \dots , x_n^n)$. Since $I(S_n,V)\subseteq I(G,V)$, we get $(x_1, x_2^2, \dots , x_n^n)\subseteq \Lead (I(G,V))$. For $1\le i\le n$  and the vector   $\underline{k}=(k_1, \dots ,k_m)\in \N^m$ with $k_1+k_2+\cdots +k_m=i$,  we consider $\Pol_{\underline{k}}(x_i^i))$. Since $x_i^i\in \Lead (I(G,V))$,  the previous proposition yields $\LeadM (\Pol_{\underline{k}}(x_i^i))\in  \Lead (I(G,V^m))$. To identify  $\LeadM (\Pol_{\underline{k}}(x_i^i))$ we write
   $\Phi(x_i^i) =   (x^{(1)}_i t_1 + \cdots + x^{(m)}_i t_m)^{i}$. So we get  $$\Pol_{\underline{k}}(x_i^i))= \frac{i!}{k_1! \cdots k_m!} \prod_{j=1}^m \left(x^{(j)}_i \right)^{k_j}.$$
   Note that since $i\le n$, the coefficient is non-zero. It follows that $\Lead (I(G,V^m))$ contains the set of monomials
   $$\{\prod_{j=1}^m \left(x^{(j)}_i \right)^{k_j} \mid 1\le i\le n, \; \sum_{j=1}^m k_j=i 
   \}.$$
   Therefore, the top degree of the coinvariants $\kk[V^m]_G$ is bounded above by ${n\choose 2}$ as well. This implies that  $I(G,V^m)$ is generated by polynomials of degree at most ${n\choose 2}+1$. Now we apply a standard argument to get a bound for  $\beta (G,V^m)$ as follows. Let $f_1, \dots , f_s$ be generators for $I(G,V^m)$ of degree at most ${n\choose 2}+1$. We may assume these generators lie in $\kk[V^m]^G$. Let $f\in \kk[V^m]^G$ with degree $>{n\choose 2}+1$. Write 
   $f=\sum_{i=1}^s q_if_i$ with $q_i\in \kk[V^m]_+$. Let $Tr:\kk[V^m]\rightarrow \kk[V^m]^G$ denote the transfer map defined by $\Tr (h)=\sum_{\sigma \in G}\sigma (h)$ for $h\in \kk[V^m] $. Then we have $\Tr (f)=|G|f=\sum_{i=1}^s \Tr(q_i)f_i.$ Therefore $f$ is in the algebra generated by invariants of strictly smaller degree. So we get   $\beta (G,V^m)\le {n\choose 2}+1$ as desired.
   \end{proof}
\bibliographystyle{plain}
\bibliography{muf}

\end{document}